\documentclass[12pt]{amsart}

\usepackage{mathrsfs,amssymb}

\newtheorem{theorem}{Theorem}[section]
\newtheorem{lemma}[theorem]{Lemma}
\newtheorem{prop}[theorem]{Proposition}
\newtheorem{cor}[theorem]{Corollary}

\theoremstyle{definition}

\theoremstyle{remark}
\newtheorem{remark}[theorem]{Remark}

\numberwithin{equation}{section}

\let \la=\lambda
\let \e=\varepsilon
\let \d=\delta
\let \o=\omega
\let \a=\alpha

\let \si=\sigma

\begin{document}

\title[On weighted norm inequalities]
{On weighted norm inequalities for the Carleson operator}

\author{Andrei K. Lerner}
\address{Department of Mathematics,
Bar-Ilan University, 5290002 Ramat Gan, Israel}
\email{aklerner@netvision.net.il}

\thanks{This research was supported by the Israel Science Foundation (grant No. 953/13).}

\begin{abstract}
We obtain $L^p(w)$ bounds for the Carleson operator ${\mathcal C}$ in terms
of the $A_q$ constants $[w]_{A_q}$ for $1\le q\le p$. In particular,
we show that, exactly as for the Hilbert transform, $\|{\mathcal C}\|_{L^p(w)}$ is
bounded linearly by $[w]_{A_q}$ for $1\le q<p$. Our approach works in the general
context of maximally modulated Calder\'on-Zygmung operators.
\end{abstract}



\maketitle

\section{Introduction}
For $f\in L^p({\mathbb R}), 1<p<\infty$, define the Carleson operator ${\mathcal C}$ by
$${\mathcal C}(f)(x)=\sup_{\xi\in {\mathbb R}}|H({\mathcal M}^{\xi}f)(x)|,$$
where $H$ is the Hilbert transform, and ${\mathcal M}^{\xi}f(x)={\rm e}^{2\pi i\xi x}f(x)$.

A famous Carleson-Hunt theorem on a.e. convergence of Fourier series in one of its equivalent
statements says that ${\mathcal C}$ is bounded on $L^p$ for any $1<p<\infty$. The crucial step was done
by Carleson \cite{C} who established that ${\mathcal C}$ maps $L^2$ into weak-$L^2$. After that Hunt \cite{H}
extended this result to any $1<p<\infty$. Alternative proofs of this theorem were obtained by Fefferman \cite{F} and by
Lacey-Thiele \cite{LT}. We refer also to \cite{Ar}, \cite[Ch. 11]{G} and \cite[Ch. 7]{MS}.

By a weight we mean a non-negative locally integrable function.
The weighted boundedness of ${\mathcal C}$ also is well known. Hunt-Young \cite{H} showed that ${\mathcal C}$ is bounded
on $L^p(w), 1<p<\infty,$ if $w$ satisfies the $A_p$ condition (see also \cite[p. 475]{G}). In \cite{GMS},
Grafakos-Martell-Soria extended this result to a more general class of maximally modulated singular integrals.
A different approach (as well as a kind of strengthening) to the Hunt-Young result was recently obtained by
Do-Lacey~\cite{DL}.

In the past decade a lot of attention was devoted to sharp $L^p(w)$ estimates in terms of the $A_p$ constants $[w]_{A_p}$.
Recall that these constants are defined as follows:
$$[w]_{A_p}=\sup_{Q}\left(\frac{1}{|Q|}\int_Qwdx\right)\left(\frac{1}{|Q|}\int_Qw^{-\frac{1}{p-1}}dx\right)^{p-1},1<p<\infty,$$
and
$$[w]_{A_1}=\sup_{Q}\left(\frac{1}{|Q|}\int_Qwdx\right)(\inf_Qw)^{-1},$$
where the supremum is taken over all cubes $Q\subset {\mathbb R}^n$.
Sharp bounds for $L^p(w)$ operator norms in terms of $[w]_{A_p}$ have been recently found
for many central operators in Harmonic Analysis (see, e.g., \cite{B,CMP,Hyt1,L2,L3,P}).
A relatively simple approach to such bounds based on local mean oscillation estimates was developed in \cite{CMP,Hyt2,L1,L2,L3}.

In this paper we apply the ``local mean oscillation estimate" approach to
the Carleson operator ${\mathcal C}$. In particular, we obtain sharp linear bounds for $\|{\mathcal C}\|_{L^p(w)}$ in terms of
$[w]_{A_q}$ for any $1\le q<p<\infty$.
Our main results can be described in the framework of maximally modulated singular integrals
studied by Grafakos-Martell-Soria \cite{GMS}.

We give several main definitions. A Calder\'on-Zygmund operator on ${\mathbb R}^n$ is an $L^2$ bounded integral operator represented as
$$Tf(x)=\int_{{\mathbb R}^n}K(x,y)f(y)dy,\quad x\not\in\text{supp}\,f,$$
with kernel $K$ satisfying the following growth and smoothness conditions:
\begin{enumerate}
\renewcommand{\labelenumi}{(\roman{enumi})}
\item
$|K(x,y)|\le \frac{c}{|x-y|^n}$ for all $x\not=y$;
\item
$|K(x,y)-K(x',y)|+|K(y,x)-K(y,x')|\le
c\frac{|x-x'|^{\d}}{|x-y|^{n+\d}}$
for some $0<\d\le 1$ when $|x-x'|<|x-y|/2$.
\end{enumerate}

Let ${\mathcal F}=\{\phi_{\a}\}_{\a\in A}$ be a family of real-valued measurable functions indexed by some set $A$, and let $T$ be a Calder\'on-Zygmund
operator. Then the maximally modulated Calder\'on-Zygmund operator $T^{\mathcal F}$ is defined by
$$T^{\mathcal F}f(x)=\sup_{\a\in A}|T({\mathcal M}^{\phi_{\a}}f)(x)|,$$
where ${\mathcal M}^{\phi_{\a}}f(x)={\rm e}^{2\pi i\phi_{\a}(x)}f(x)$.

As it was shown in \cite{GMS}, the weighted theory of such operators can be developed under a single {\it a priori} assumption on $T^{\mathcal F}$.
We state this assumption as follows. Let $\Phi$ be a Young function, that is, $\Phi:[0,\infty)\to [0,\infty)$, $\Phi$ is continuous,
convex, increasing, $\Phi(0)=0$ and $\Phi(t)\to \infty$ as $t\to \infty$. Define the mean Luxemburg norm of $f$ on a cube $Q\subset {\mathbb R}^n$ by
$$\|f\|_{\Phi,Q}=\inf\left\{\la>0:\frac{1}{|Q|}\int_Q\Phi\left(\frac{|f(x)|}{{\la}}\right)dx\le 1\right\}.$$
Our basic assumption on $T^{\mathcal F}$ is the following: for any cube $Q\subset {\mathbb R}^n$,
\begin{equation}\label{cond}
\|T^{\mathcal F}(f\chi_Q)\|_{L^{1,\infty}(Q)}\lesssim |Q|\|f\|_{\Phi,Q}.
\end{equation}

If $\phi_{\a}(x)=0$, then $T^{\mathcal F}=T$ is the usual Calder\'on-Zygmund operator, and in this case (\ref{cond}) holds with $\Phi(t)=t$,
which corresponds to the weak type $(1,1)$ of $T$. Suppose that $n=1, \phi_{\a}(x)=\a x$ and $A={\mathbb R}$. Then $T^{\mathcal F}={\mathcal C}$ is the Carleson
operator, and the currently best known result is that (\ref{cond}) holds with
$\Phi(t)=t\log({\rm e}+t)\log\log\log({\rm e}^{{\rm e}^{\rm e}}+t),$
see \cite[Th. 5.1]{GMS}. This represents an elaborated version of Antonov's theorem~\cite{A} on a.e. convergence of Fourier series for $f\in L\log L\log\log\log L$ (see also \cite{SS}). For other examples concerning (\ref{cond}) we refer to \cite{GMS}.

Assuming (\ref{cond}), it is easy to show that $T^{\mathcal F}$ is controlled (either via a good-$\lambda$ inequality or by a sharp function estimate)
by the Orlicz maximal function $M_{\Phi}$ defined by
$$M_{\Phi}f(x)=\sup_{Q\ni x}\|f\|_{\Phi,Q}.$$
Since we are interested in $L^p(w)$ estimates for $T^{\mathcal F}$ with $w\in A_{p}$, it is assumed implicitly (by the Rubio de Francia extrapolation theorem)
that $M_{\Phi}$ (and so $T^{\mathcal F}$) is bounded on the unweighted $L^p$ for any $p>1$. It was shown by P\'erez \cite{Pe} that $M_{\Phi}$ is bounded on $L^p$ if and only if $\Phi$ satisfies the
$B_p$ condition: $\int_1^{\infty}\Phi(t)t^{-p-1}dt<\infty.$ Therefore, throughout the paper, we assume that for any $r>1$,
$$
t\le \Phi(t)\le c_rt^r \quad (t\ge 0).
$$
This condition includes all main cases of interest. Also we introduce the following notation for the $B_p$ constant of $\Phi$:
$$C_{\Phi}(p)=\Big(\int_{1}^{\infty}\frac{\Phi(t)}{t^p}\frac{dt}{t}\Big)^{1/p}.$$

Before stating our main results about $T^{\mathcal F}$,
we summarize below sharp weighted bounds for standard Calder\'on-Zygmund operators.

\vskip 2mm
\noindent
{\bf Theorem A.} Let $T$ be a Calder\'on-Zygmund operator on ${\mathbb R}^n$.
\begin{enumerate}
\renewcommand{\labelenumi}{(\roman{enumi})}
\item For any $1\le q<p<\infty$,
$$\|T\|_{L^p(w)}\le c(n,T,q,p)[w]_{A_q},$$
and in the case $q=1$, $c(n,T,1,p)=c(n,T)pp'$;
\item for any $1<p<\infty$,
$$\|T\|_{L^p(w)}\le c(n,T,p)[w]_{A_p}^{\max\big(1,\frac{1}{p-1}\big)}.$$
\end{enumerate}

Part (i) for $q=1$ was obtained by Lerner-Ombrosi-Perez \cite{LOP1,LOP2}, and later Duoandikoetxea \cite{D} showed that the result for $q=1$ can be
self-improved by extrapolation to any $1<q<p$. The sharp dependence of $c(n,T,1,p)$ on $p$ is important for a weighted weak $L^1$ bound of $T$ in terms
of $[w]_{A_1}$ \cite{LOP2}.
Part (ii) (known as the $A_2$ conjecture) is a more difficult result. First it was proved by
Petermichl \cite{P} for the Hilbert transform, and recently Hyt\"onen \cite{Hyt1} obtained (ii) for general Calder\'on-Zygmund operators.
A proof of Theorem A based on local mean oscillation estimates was found in \cite{L3,L4}. Observe that
for $p\ge 2$, (i) follows from (ii) but for $1<p<2$, (i) and (ii) are independent results.

Denote by $S_0({\mathbb R}^n)$ the class of measurable functions on ${\mathbb R}^n$ such that
$$\mu_{f}(\la)=|\{x\in {\mathbb R}^n:|f(x)|>\la\}|<\infty$$
for all $\la>0$. Our main results are the following.

\begin{theorem}\label{mainr} Let $T^{\mathcal F}$ be a maximally modulated Calder\'on-Zygmund operator
satisfying (\ref{cond}).
\begin{enumerate}
\renewcommand{\labelenumi}{(\roman{enumi})}
\item For any $1\le q<p<\infty$,
$$\|T^{\mathcal F}f\|_{L^p(w)}\le c(n,T,q,p)[w]_{A_q}\|f\|_{L^p(w)},$$
and in the case $q=1$, $c(n,T,1,p)=c(n,T)pC_{\Phi}\Big(\frac{p+1}{2}\Big)$.
\item Assume that $\Phi(t)\simeq t\int_1^t\frac{\Psi(u)}{u^2}du$ for $t\ge c_0$, where $\Psi$ is a Young function.
Then for any $1<p<\infty$,
$$\|T^{\mathcal F}f\|_{L^p(w)}\le c(n,T,p)[w]_{A_p}^{\max\big(1,\frac{1}{p-1}\big)+\frac{1}{p}}C_{\Psi^{p-\e}}(p)\|f\|_{L^p(w)},$$
where $\e\simeq [w]_{A_p}^{1-p'}$.
\end{enumerate}
Both estimates in (i) and (ii) are understood in the sense that they hold for any $f\in L^p(w)$ for which
$T^{\mathcal F}f\in S_0$.
\end{theorem}

Several remarks about Theorem \ref{mainr} are in order.

\begin{remark}\label{rem1}
The last sentence in Theorem \ref{mainr} can be removed if it is additionally known that $T^{\mathcal F}f\in S_0$ for some dense subset in $L^p(w)$, for instance, for
Schwartz functions. In particular, this obviously holds if $T^{\mathcal F}$ is of weak type $(r_0,r_0)$ for some $r_0>1$. Hence, there is no need in the last sentence
in Theorem \ref{mainr} for the Carleson operator.
\end{remark}

\begin{remark}\label{rem2}
It is easy to see that if $\Phi(t)=t$, then $C_{\Phi}\Big(\frac{p+1}{2}\Big)\simeq p'$,
and hence part (i) of Theorem \ref{mainr} contains part (i) of Theorem A as a particular case. On the other hand,
part (ii) of Theorem \ref{mainr} does not contain Theorem A, since the assumption $\Phi(t)\simeq t\int_1^t\frac{\Psi(u)}{u^2}du$
implies $t\log(\rm{e}+t)\lesssim \Phi(t)$.
\end{remark}

\begin{remark}\label{rem3}
Consider the case corresponding to the Carleson operator, namely, assume that
$\Phi(t)=t\log({\rm e}+t)\log\log\log({\rm e}^{{\rm e}^{\rm e}}+t)$. Simple computations show that
in this case,
$$C_{\Phi}\Big(\frac{p+1}{2}\Big)\simeq \frac{p^2}{(p-1)^2}\log\log\Big(\rm{e}^{\rm{e}}+\frac{1}{p-1}\Big).$$
Concerning part (ii), it is easy to see that $\Psi(t)\simeq t\log\log\log ({\rm e}^{{\rm e}^{\rm e}}+t)$ and
$C_{\Psi^{p-\e}}(p)\simeq \frac{1}{\e^{1/p}}\log\log({\rm e}^{\rm e}+1/\e)$. Therefore, if $\e\simeq [w]_{A_p}^{1-p'}$,
then
$$C_{\Psi^{p-\e}}(p)\simeq [w]_{A_p}^{\frac{1}{p(p-1)}}\log\log({\rm e}^{\rm e}+[w]_{A_p}).$$
\end{remark}

Thus, we obtain the following corollary from Theorem \ref{mainr}.

\begin{cor}\label{carl} Let $\mathcal C$ be the Carleson operator.
\begin{enumerate}
\renewcommand{\labelenumi}{(\roman{enumi})}
\item For any $1\le q<p<\infty$,
$$\|\mathcal C\|_{L^p(w)}\le c(q,p)[w]_{A_q},$$
and in the case $q=1$, $c(1,p)\simeq \frac{p^3}{(p-1)^2}\log\log\Big(\rm{e}^{\rm{e}}+\frac{1}{p-1}\Big)$;
\item for any $1<p<\infty$,
$$\|\mathcal C\|_{L^p(w)}\le c(p)[w]_{A_p}^{\max\big(p',\frac{2}{p-1}\big)}\log\log ({\rm e}^{\rm e}+[w]_{A_p}).$$
\end{enumerate}
\end{cor}

We make several additional remarks.

\begin{remark}\label{rem4}
Since the linear $[w]_{A_q}, 1\le q<p,$ bound is sharp for the Hilbert transform, it is obviously sharp also for ${\mathcal C}$. Further,
observe that, as soon as we know, even in the unweighted case $C(1,p)$ from (i) is the currently best known bound for $\|{\mathcal C}\|_{L^p}$
when $p\to 1$. We could not find in the literature this bound written explicitly but it is apparently well known. In particular, it can be easily
deduced from a good-$\lambda$ inequality related ${\mathcal C}$ and $M_{\Phi}$ with $\Phi(t)=t\log({\rm e}+t)\log\log\log({\rm e}^{{\rm e}^{\rm e}}+t)$
obtained in \cite{GMS}.

Concerning the bound for  $\|\mathcal C\|_{L^p(w)}$ in terms of $[w]_{A_p}$ in (ii), most probably it is not sharp. We discuss this point in Section 4 below.
\end{remark}

\begin{remark}\label{rem5}
The key ingredient in the proof of the linear $[w]_{A_1}$ bound for usual Calder\'on-Zygmund operators $T$ in \cite{LOP1,LOP2}
is a Coifman type estimate relating the adjoint operator $T^*$ and the Hardy-Littlewood maximal operator $M$. It was crucial
that $T^*$ is essentially the same operator as $T$. However, this is not the case with the Carleson operator ${\mathcal C}$. Indeed,
taking an arbitrary measurable function $\xi(\cdot)$, we can consider the standard linearization of ${\mathcal C}$ given by
$${\mathcal C}_{\xi(\cdot)}(f)(x)=H({\mathcal M}^{\xi(x)}f)(x).$$
It is difficult to expect that its adjoint ${\mathcal C}_{\xi(\cdot)}^*$ can be related (uniformly in $\xi(\cdot)$) with $M$ (or even with a bigger maximal operator)
either via good-$\lambda$ or by a sharp function estimate. Indeed, such a relation would imply that $\|{\mathcal C}_{\xi(\cdot)}^*\|_{L^p}\lesssim p$
as $p\to\infty$ (since $\|f\|_{L^p}\lesssim p\|f^{\#}\|_{L^p}$ as $p~\to~\infty$, where $f^{\#}$ is the sharp function), which in turn means that $\|{\mathcal C}\|_{L^p}~\lesssim~ \frac{1}{p-1}$ as $p\to 1$. But due to the previous remark, the currently
known behavior of $\|{\mathcal C}\|_{L^p}$ is far from $\frac{1}{p-1}$ for $p$ is close to 1 (in fact, it is reasonable to conjecture that the best possible bound for
$\|{\mathcal C}\|_{L^p}$ when $p$ is close to~1 is $\frac{1}{(p-1)^2}$, see a relevant discussion in Section 4).

In order to prove the linear $[w]_{A_1}$ bound for ${\mathcal C}$, we use a modified approach based partially on ideas from \cite{L1} and \cite{LOP1}.
\end{remark}

The paper is organized as follows. In Section 2, we obtain a local mean oscillation estimate of $T^{\mathcal F}$, and the corresponding bound by
dyadic sparse operators. Using this result, we prove Theorem \ref{mainr} in Section~3. In Section 4, we discuss a connection between the $L\log L$
conjecture about a.e. convergence of Fourier series and sharp $L^p(w)$ bounds for ${\mathcal C}$ in terms of $[w]_{A_p}$.

Throughout the paper, we use the notation $A\lesssim B$ to indicate that there is a constant $c$, independent of the important parameters, such that $A\leq cB$.
We write $A\simeq B$ when $A\lesssim B$ and $B\lesssim A$.

\vskip 2mm
{\bf Acknowledgement.} I am very grateful to Loukas Grafakos for his useful comments
on the Carleson operator.

\section{An estimate of $T^{\mathcal F}$ by dyadic sparse operators}
\subsection{A local mean oscillation estimate}
By a general dyadic grid ${\mathscr{D}}$ we mean a collection of
cubes with the following properties: (i)
for any $Q\in {\mathscr{D}}$ its sidelength $\ell_Q$ is of the form
$2^k, k\in {\mathbb Z}$; (ii) $Q\cap R\in\{Q,R,\emptyset\}$ for any $Q,R\in {\mathscr{D}}$;
(iii) the cubes of a fixed sidelength $2^k$ form a partition of ${\mathbb
R}^n$.

Denote the standard dyadic grid $\{2^{-k}([0,1)^n+j), k\in{\mathbb Z}, j\in{\mathbb Z}^n\}$
by ${\mathcal D}$.
Given a cube $Q_0$, denote by ${\mathcal D}(Q_0)$ the set of all
dyadic cubes with respect to $Q_0$, that is, the cubes from ${\mathcal D}(Q_0)$ are formed
by repeated subdivision of $Q_0$ and each of its descendants into $2^n$ congruent subcubes.

We say that a family of cubes ${\mathcal S}$ is sparse if for any cube $Q\in {\mathcal S}$ there is a
measurable subset $E(Q)\subset Q$ such that $|Q|\le 2|E(Q)|$, and the sets $\{E(Q)\}_{Q\in {\mathcal S}}$
are pairwise disjoint.

Given a measurable function $f$ on ${\mathbb R}^n$ and a cube $Q$,
the local mean oscillation of $f$ on $Q$ is defined by
$$\o_{\la}(f;Q)=\inf_{c\in {\mathbb R}}
\big((f-c)\chi_{Q}\big)^*\big(\la|Q|\big)\quad(0<\la<1),$$
where $f^*$ denotes the non-increasing rearrangement of $f$.

By a median value of $f$ over $Q$ we mean a possibly nonunique, real
number $m_f(Q)$ such that
$$\max\big(|\{x\in Q: f(x)>m_f(Q)\}|,|\{x\in Q: f(x)<m_f(Q)\}|\big)\le |Q|/2.$$

The following result was proved in \cite{L1}; in its current refined version given below it can be found in \cite{Hyt2}.

\begin{theorem}\label{lmoes} Let $f$ be a measurable function on
${\mathbb R}^n$ and let $Q_0$ be a fixed cube. Then there exists a
(possibly empty) sparse family ${\mathcal S}$ of cubes from ${\mathcal D}(Q_0)$ such that for a.e. $x\in Q_0$,
$$
|f(x)-m_f(Q_0)|\le 2\sum_{Q\in {\mathcal S}}
\o_{\frac{1}{2^{n+2}}}(f;Q)\chi_{Q}(x).
$$
\end{theorem}

\subsection{An application to $T^{\mathcal F}$} We now apply Theorem \ref{lmoes} to $T^{\mathcal F}$.
Given a cube $Q$, we denote $\bar Q=2\sqrt n Q$.

\begin{lemma}\label{oscest}
Suppose $T^{\mathcal F}$ satisfies (\ref{cond}). Then for any cube $Q\subset {\mathbb R}^n$,
\begin{equation}\label{osc}
\o_{\la}(T^{\mathcal F}f;Q)\lesssim \|f\|_{\Phi,\bar Q}+
\sum_{m=0}^{\infty}\frac{1}{2^{m\d}}\left(\frac{1}{|2^mQ|}\int_{2^mQ}|f(y)|dy\right).
\end{equation}
\end{lemma}

\begin{proof}
This result is a minor modification of \cite[Prop. 2.3]{L3}, and it is essentially contained in \cite[Prop. 4.1]{GMS}.
We outline briefly main details.

Observe that (\ref{cond}) can be written in an equivalent form:
$$
\big(T^{\mathcal F}(f\chi_Q)\chi_Q\big)^*(t)\lesssim \frac{1}{t}|Q|\|f\|_{\Phi,Q},
$$
which implies
\begin{equation}\label{cond1}
\big(T^{\mathcal F}(f\chi_Q)\chi_Q\big)^*(\la|Q|)\lesssim \|f\|_{\Phi,Q}.
\end{equation}

Set $f_1=f\chi_{\bar Q}$ and $f_2=f-f_1$. Let $x\in Q$ and let $x_0$ be the center of $Q$. Then
\begin{eqnarray*}
&&|T^{\mathcal F}(f)(x)-T^{\mathcal F}(f_2)(x_0)|\\
&&=\Big|\sup_{\a\in A}
|T({\mathcal M}^{\phi_{\a}}f)(x)|-\sup_{\a\in A}|T({\mathcal M}^{\phi_{\a}}f_2)(x_0)|\Big|\\
&&\le \sup_{\a\in A}|T({\mathcal M}^{\phi_{\a}}f)(x)-T({\mathcal M}^{\phi_{\a}}f_2)(x_0)|\\
&&\le T^{\mathcal F}(f_1)(x)+\sup_{\a\in A}\|T({\mathcal M}^{\phi_{\a}}f_2)(\cdot)-T({\mathcal M}^{\phi_{\a}}f_2)(x_0)\|_{L^{\infty}(Q)}.
\end{eqnarray*}

Exactly as in \cite[Prop. 2.3]{L3}, by the kernel assumption,
\begin{eqnarray*}
&&\sup_{\a\in A}\|T({\mathcal M}^{\phi_{\a}}f_2)(\cdot)-T({\mathcal M}^{\phi_{\a}}f_2)(x_0)\|_{L^{\infty}(Q)}\\
&&\le \int_{{\mathbb R}^n\setminus \bar Q}|f(y)|\|K(\cdot,y)-K(x_0,y)\|_{L^{\infty}(Q)}dy\\
&&\lesssim \sum_{m=0}^{\infty}\frac{1}{2^{m\d}}\left(\frac{1}{|2^mQ|}\int_{2^mQ}|f(y)|dy\right).
\end{eqnarray*}
For the local part, by (\ref{cond1}),
$$\big(T^{\mathcal F}(f_1)\chi_Q\big)^*(\la|Q|)\lesssim \|f\|_{\Phi,\bar Q}.$$
Combining this estimate with the two previous ones, and taking $c=T^{\mathcal F}(f_2)(x_0)$ in the definition of $\o_{\la}(T^{\mathcal F}f;Q)$
proves (\ref{osc}).
\end{proof}

Given a sparse family ${\mathcal S}$, define the operators ${\mathcal A}_{\Phi, \mathcal S}$ and ${\mathcal T}_{\mathcal S,m}$
respectively by
$${\mathcal A}_{\Phi, \mathcal S}f(x)=\sum_{Q\in {\mathcal S}}\|f\|_{\Phi, \bar Q}\chi_Q(x)$$
and
$$
{\mathcal T}_{\mathcal S,m}f(x)=\sum_{Q\in {\mathcal S}}|f|_{2^mQ}\chi_Q(x)
$$
(we use a standard notation $f_Q=\frac{1}{|Q|}\int_Qf$).

\begin{lemma}\label{kest} Suppose $T^{\mathcal F}$ satisfies (\ref{cond}). Let $1<p<\infty$ and let $w$ be an arbitrary weight. Then
\begin{equation}\label{right}
\|T^{\mathcal F}f\|_{L^p(w)}\lesssim \sup_{{\mathscr{D}},{\mathcal S}}\|{\mathcal A}_{\Phi, \mathcal S}f\|_{L^p(w)}
\end{equation}
for any $f$ for which $T^{\mathcal F}f\in S_0$, where the supremum is taken over all dyadic grids ${\mathscr{D}}$ and all
sparse families ${\mathcal S}\subset {\mathscr{D}}$.
\end{lemma}

\begin{proof} Let $Q_0\in {\mathcal D}$. Combining Theorem \ref{lmoes} with Lemma \ref{oscest}, we obtain that there exists a sparse family ${\mathcal S}\subset {\mathcal D}$ such that for a.e. $x\in Q_0$,
\begin{equation}\label{intes}
|T^{\mathcal F}f(x)-m_{T^{\mathcal F}f}(Q_0)|\lesssim {\mathcal A}_{\Phi, \mathcal S}f(x)+\sum_{m=0}^{\infty}\frac{1}{2^{m\d}}{\mathcal T}_{\mathcal S,m}f(x).
\end{equation}

If $T^{\mathcal F}f\in S_0$, then $m_{T^{\mathcal F}f}(Q)\to 0$ as $|Q|\to \infty$. Hence, letting $Q_0$ to anyone of $2^n$ quadrants and
using (\ref{intes}) along with Fatou's lemma, we get
$$\|T^{\mathcal F}f\|_{L^p(w)}\lesssim \sup_{{\mathcal S}\subset {\mathcal D}}\|{\mathcal A}_{\Phi, \mathcal S}f\|_{L^p(w)}+
\sum_{m=0}^{\infty}\frac{1}{2^{m\d}}\sup_{{\mathcal S}\subset {\mathcal D}}\|{\mathcal T}_{\mathcal S,m}f\|_{L^p(w)}.$$
It was shown in \cite{L3} that
$$\sup_{{\mathcal S}\subset {\mathcal D}}\|{\mathcal T}_{\mathcal S,m}f\|_{L^p(w)}\lesssim
m\sup_{{\mathscr{D}},{\mathcal S}}\|{\mathcal T}_{\mathcal S,0}f\|_{L^p(w)}.
$$
Since $t\le \Phi(t)$, we have $|f|_Q\lesssim \|f\|_{\Phi, \bar Q}$, and hence
$$\|{\mathcal T}_{\mathcal S,0}f\|_{L^p(w)}\lesssim \|{\mathcal A}_{\Phi, \mathcal S}f\|_{L^p(w)}.$$
Combining this with the two previous estimates completes the proof.
\end{proof}

\begin{remark}\label{rem}
Observe that the implicit constant in (\ref{right}) depends only on $T^{\mathcal F}$ and $n$. In fact, (\ref{right}) holds with an arbitrary Banach function
space $X$ instead of $L^p(w)$ exactly as for standard Calder\'on-Zygmund operators (see \cite{L3}).
\end{remark}

\section{Proof of Theorem \ref{mainr}}
\subsection{Proof of Theorem \ref{mainr}, part (i)}
We start with some preliminaries. Given a Young function $\Phi$, the complementary Young function $\bar\Phi$ is defined by
$$\bar\Phi(t)=\sup_{s>0}\{st-\Phi(s)\}.$$
A well known result about the equivalence of Orlicz and Luxemburg norms (see, e.g., \cite[Th. 8.14]{BS}) says that
\begin{equation}\label{eq}
\|f\|_{\Phi,Q}\le \sup_{g:\|g\|_{\bar \Phi,Q}\le 1}\Big|\frac{1}{|Q|}\int_Qfgdx\Big|\le 2\|f\|_{\Phi,Q}.
\end{equation}

For $r>0$ let $M_rf(x)=M(|f|^r)(x)^{1/r}$, where $M$ is the Hardy-Littlewood maximal operator.
We summarize below several results from \cite{LOP1} (notice that part (ii) is contained in the proof of \cite[Lemma 3.3]{LOP1}).

\newpage
\begin{prop}\label{sum} The following estimates hold:
\begin{enumerate}
\renewcommand{\labelenumi}{(\roman{enumi})}
\item if $w\in A_1$ and $r_w=1+\frac{1}{2^{n+1}[w]_{A_1}}$, then
$$M_{r_{w}}w(x)\le 2[w]_{A_1}w(x);$$
\item for any $p>1$ and $1<r<2$,
$$\|Mf\|_{L^{p'}((M_rw)^{-\frac{1}{p-1}})}\le c(n)p\Big(\frac{1}{r-1}\Big)^{1-1/pr}\|f\|_{L^{p'}(w^{-\frac{1}{p-1}})}.$$
\end{enumerate}
\end{prop}

Also we use the following generalization of the classical Fefferman-Stein inequality \cite{FS}
obtained by P\'erez \cite{Pe}: if $p>1$ and $\Phi\in B_p$, then for any weight $w$,
\begin{equation}\label{per}
\|M_{\Phi}f\|_{L^p(w)}\le c(n)C_{\Phi}(p)\|f\|_{L^p(Mw)}.
\end{equation}

\begin{proof}[Proof of Theorem \ref{mainr}, part (i)]
By extrapolation (\cite[Cor. 4.3.]{D}), it suffices to consider only the case $q=1$. Hence, our aim is to show
that for any $1<p<\infty$,
$$\|T^{\mathcal F}f\|_{L^p(w)}\le c(n,T)pC_{\Phi}\Big(\frac{p+1}{2}\Big)[w]_{A_1}\|f\|_{L^p(w)}.$$
By Lemma \ref{kest} (see also Remark \ref{rem}), this would follow from
\begin{equation}\label{a1norm}
\sup_{{\mathscr{D}},{\mathcal S}}\|{\mathcal A}_{\Phi, \mathcal S}f\|_{L^p(w)}\le c(n)pC_{\Phi}\Big(\frac{p+1}{2}\Big)[w]_{A_1}\|f\|_{L^p(w)}.
\end{equation}

Fix a dyadic grid ${\mathscr{D}}$ and a sparse family ${\mathcal S}\subset {\mathscr{D}}$. Using (\ref{eq}), we linearize the operator ${\mathcal A}_{\Phi, \mathcal S}$. One can assume that $f\ge 0$. For any $Q\in {\mathcal S}$ there exists $g_{(Q)}$ supported in $\bar Q$ such that $\|g_{(Q)}\|_{\bar \Phi,\bar Q}\le 1$ and
$$
\|f\|_{\Phi,\bar Q}\le (fg_{(Q)})_{\bar Q}.
$$
Define now a linear operator
$$L(h)(x)=\sum_{Q\in {\mathcal S}}(hg_{(Q)})_{\bar Q}\chi_Q(x).$$
Then in order to prove (\ref{a1norm}), it suffices to show that
\begin{equation}\label{suff}
\|L(h)\|_{L^p(w)}\le c(n)pC_{\Phi}\Big(\frac{p+1}{2}\Big)[w]_{A_1}\|h\|_{L^p(w)},
\end{equation}
uniformly in $g_{(Q)}$.

Exactly as it was done in \cite{LOP1}, we have that (\ref{suff}) will follow from
\begin{equation}\label{suff1}
\|L(h)\|_{L^p(w)}\le c(n)pC_{\Phi}\Big(\frac{p+1}{2}\Big)
\Big(\frac{1}{r-1}\Big)^{1-1/pr}\|h\|_{L^p(M_rw)},
\end{equation}
where $1<r<2$. Indeed, taking here $r=r_w=1+\frac{1}{2^{n+1}[w]_{A_1}}$, by (i) of Proposition \ref{sum},
$$
\Big(\frac{1}{r_w-1}\Big)^{1-1/pr_w}\|h\|_{L^p(M_{r_w}w)}\le c(n)[w]_{A_1}\|h\|_{L^p(w)},
$$
which yields (\ref{suff}).

Let $L^*$ denote the formal adjoint of $L$. By duality, (\ref{suff1}) is equivalent to
$$
\|L^*(h)\|_{L^{p'}((M_rw)^{-\frac{1}{p-1}})}\le
c(n)pC_{\Phi}\Big(\frac{p+1}{2}\Big)
\Big(\frac{1}{r-1}\Big)^{1-1/pr}\|h\|_{L^{p'}(w^{-\frac{1}{p-1}})},
$$
which, by (ii) of Proposition \ref{sum}, is an immediate corollary of
\begin{equation}\label{suff2}
\|L^*(h)\|_{L^{p'}((M_rw)^{-\frac{1}{p-1}})}\le c(n)C_{\Phi}\Big(\frac{p+1}{2}\Big)\|Mh\|_{L^{p'}((M_rw)^{-\frac{1}{p-1}})}.
\end{equation}

We now prove (\ref{suff2}).
By duality, pick $\eta\ge 0$ such that $\|\eta\|_{L^p(M_rw)}=1$ and
$$\|L^*(h)\|_{L^{p'}((M_rw)^{-\frac{1}{p-1}})}=\int_{{\mathbb R}^n}L^*(h)\eta dx=\int_{{\mathbb R}^n}hL(\eta) dx.$$
Applying (\ref{eq}) again, we get
\begin{eqnarray*}
\int_{{\mathbb R}^n}hL(\eta) dx&=&\sum_{Q\in {\mathcal S}}(\eta g_{(Q)})_{\bar Q}\int_Qh\le
2\sum_{Q\in {\mathcal S}}\|\eta\|_{\Phi,\bar Q}\int_Qh\\
&\le& 2(2\sqrt n)^n\sum_{Q\in {\mathcal S}}\|\eta\|_{\Phi,\bar Q}h_{\bar Q}|Q|\\
&\le& 4(2\sqrt n)^n \sum_{Q\in {\mathcal S}} \|(Mh)^{\frac{1}{p+1}}\eta\|_{\Phi,\bar Q}(h_{\bar Q})^{\frac{p}{p+1}}|E(Q)|\\
&\le& 4(2\sqrt n)^n\sum_{Q\in {\mathcal S}}\int_{E(Q)}M_{\Phi}((Mh)^{\frac{1}{p+1}}\eta)(Mh)^{\frac{p}{p+1}}dx\\
&\le& 4(2\sqrt n)^n \int_{{\mathbb R}^n}M_{\Phi}((Mh)^{\frac{1}{p+1}}\eta)(Mh)^{\frac{p}{p+1}}dx.
\end{eqnarray*}
Next, by H\"older's inequality with the exponents $s=\frac{p+1}{2}$ and $s'=\frac{p+1}{p-1}$,
\begin{eqnarray*}
&&\int_{{\mathbb R}^n}M_{\Phi}((Mh)^{\frac{1}{p+1}}\eta)(Mh)^{\frac{p}{p+1}}dx\\
&&=\int_{{\mathbb R}^n}M_{\Phi}((Mh)^{\frac{1}{p+1}}\eta)(M_rw)^{\frac{1}{p+1}}(Mh)^{\frac{p}{p+1}}(M_rw)^{-\frac{1}{p+1}}dx\\
&&\le \|M_{\Phi}((Mh)^{\frac{1}{p+1}}\eta)\|_{L^{\frac{p+1}{2}}((M_rw)^{1/2})}\|Mh\|_{L^{p'}((M_rw)^{-\frac{1}{p-1}})}^{\frac{p}{p+1}}.
\end{eqnarray*}
Further, we apply (\ref{per}) along with Coifman's inequality \cite{CR} saying that $M(M_rw)^{1/2}\le c(n)(M_rw)^{1/2}$. We obtain
\begin{eqnarray*}
&&\|M_{\Phi}((Mh)^{\frac{1}{p+1}}\eta)\|_{L^{\frac{p+1}{2}}((M_rw)^{1/2})}\\
&&\le c(n)C_{\Phi}\Big(\frac{p+1}{2}\Big)
\|(Mh)^{\frac{1}{p+1}}\eta\|_{L^{\frac{p+1}{2}}(M(M_rw)^{1/2})}\\
&&\le c(n)C_{\Phi}\Big(\frac{p+1}{2}\Big)\|(Mh)^{\frac{1}{p+1}}\eta\|_{L^{\frac{p+1}{2}}((M_rw)^{1/2})}.
\end{eqnarray*}
Using again H\"older's inequality with $s=2p'$ and $s'=\frac{2p}{p+1}$ gives
\begin{eqnarray*}
&&\|(Mh)^{\frac{1}{p+1}}\eta\|_{L^{\frac{p+1}{2}}((M_rw)^{1/2})}\\
&&=\left(\int_{{\mathbb R}^n}\Big((Mh)^{\frac{1}{2}}(M_rw)^{-\frac{1}{2p}}\Big)\Big(\eta^{\frac{p+1}{2}}(M_rw)^{\frac{p+1}{2p}}\Big)dx\right)^{\frac{2}{p+1}}\\
&&\le \|Mh\|_{L^{p'}((M_rw)^{-\frac{1}{p-1}})}^{\frac{1}{p+1}}\|\eta\|_{L^p(M_rw)}=\|Mh\|_{L^{p'}((M_rw)^{-\frac{1}{p-1}})}^{\frac{1}{p+1}}.
\end{eqnarray*}
Combining this estimate with the three previous ones yields (\ref{suff2}), and therefore the theorem is proved.
\end{proof}

\begin{remark}\label{aboutin}
Inequality (\ref{suff2}) looks exactly as a Coifman type estimate relating $L^*$ and $M$.
However, we do not know whether there is a good-$\lambda$ inequality related $L^*$ and $M$ by
the reasons described in Remark \ref{rem5}.
\end{remark}

\subsection{A Buckley type result for $M_{\Phi}$} In order to prove the second part of Theorem \ref{mainr},
we need an extension of Buckley's bound \cite{B}:
\begin{equation}\label{bu}
\|M\|_{L^p(w)}\le c(p,n)[w]_{A_p}^{\frac{1}{p-1}}\quad (1<p<\infty)
\end{equation}
to Orlicz maximal functions $M_{\Phi}$ with general $\Phi$. In the recent work~\cite{LPR},
the case $\Phi(t)=t\log^{\la}({\rm e}+t), \la\ge 0,$ was considered:
$$\|M_{L(\log L)^{\la}}\|_{L^p(w)}\le c(p,n)[w]_{A_p}^{\frac{1+\la}{p-1}}\quad (1<p<\infty).$$
Observe that the proof in \cite{LPR} essentially contains an estimate for general $\Phi$ as stated
below in Theorem \ref{buck}. For the sake of completeness we give a somewhat different proof avoiding certain
details in \cite{LPR} (such as extrapolation). As we will see below, our proof is a direct
generalization of Buckley's proof of (\ref{bu}).

\begin{theorem}\label{buck} For all $p>1$ and any $w\in A_p$,
$$\|M_{\Phi}\|_{L^p(w)}\le c(p,n)[w]_{A_p}^{\frac{1}{p}}C_{\Phi^{p-\e}}(p),$$
where $\e\simeq [w]_{A_p}^{1-p'}$.
\end{theorem}

\begin{proof} Given a cube $Q$, define the weighted mean Luxemburg norm
$$\|f\|_{\Phi,Q}^{w}=\inf\left\{\a>0:\frac{1}{w(Q)}\int_Q\Phi\left(\frac{|f(x)|}{\a}\right)wdx\le 1\right\},$$
and consider the weighted centered Orlicz maximal function $M_{\Phi,w}^c$ defined by
$$
M_{\Phi,w}^cf(x)=\sup_{Q\ni x}\|f\|_{\Phi,Q}^{w},
$$
where the supremum is taken over all cubes $Q$ centered at $x$ (similarly we denote by $M_{\Phi}^cf$ the unweighted centered maximal function).
Then we have the following version of (\ref{per}): for any weight $w$ and all $p>1$,
\begin{equation}\label{pw}
\|M_{\Phi,w}^cf\|_{L^p(w)}\le c(n)C_{\Phi}(p)\|f\|_{L^p(w)}.
\end{equation}
The proof follows exactly the same lines as the proof of the unweighted version in
\cite{Pe} (only one should apply the Besicovitch covering theorem to get a weak type bound)
and hence we omit details.

For any $\a>0$, by H\"older's inequality,
$$\frac{1}{|Q|}\int_Q\Phi(|f|/\a)dx\le [w]_{A_p}^{\frac{1}{p}}\Big(\frac{1}{w(Q)}\int_Q\Phi(|f|/\a)^pwdx\Big)^{1/p},$$
which implies
$$\|f\|_{\Phi,Q}\le \|f\|_{[w]_{A_p}\Phi^p,Q}^{w}.$$
From this and from the standard estimate $M_{\Phi}f(x)\le M_{c(n)\Phi}^cf(x)$ we obtain
$$M_{\Phi}f(x)\le M_{[w]_{A_p}(c(n)\Phi)^p,w}^cf(x).$$
Now we use the fact that if $\e\simeq [w]_{A_p}^{1-p'}$, then $w\in A_{p-\e}$ and $[w]_{A_{p-\e}}\lesssim [w]_{A_p}$ (see \cite{B}).
Combining this with the previous estimate yields
$$M_{\Phi}f(x)\le M_{c(n,p)[w]_{A_p}\Phi^{p-\e},w}^cf(x).$$
Therefore, by (\ref{pw}),
\begin{eqnarray*}
\|M_{\Phi}f\|_{L^p(w)}&\le& \|M_{c(n,p)[w]_{A_p}\Phi^{p-\e},w}^cf\|_{L^p(w)}\\
&\le& c(n)C_{c(n,p)[w]_{A_p}\Phi^{p-\e}}(p)\|f\|_{L^p(w)}\\
&=&c(n)c(n,p)^{\frac{1}{p}}[w]_{A_p}^{\frac{1}{p}}C_{\Phi^{p-\e}}(p)\|f\|_{L^p(w)},
\end{eqnarray*}
which completes the proof.
\end{proof}

\subsection{Proof of Theorem \ref{mainr}, part (ii)}
We will need a generalization of the classical equivalence \cite{St}
$$\frac{1}{|Q|}\int_QM(f\chi_Q)dx\simeq \|f\|_{L\log L,Q}$$
to general Young functions. This can be stated as follows.

Given a Young function $\Psi$, define
\begin{equation}\label{psistar}
\Psi^{\star}(t)=\begin{cases} t, & 0\le t\le 1\\
t+t\int_1^{t}\frac{\Psi(u)}{u^2}du, & t>1.\end{cases}
\end{equation}
Then (see \cite[Theorems 10.5,10.6]{W})
\begin{equation}\label{equiv}
\frac{1}{|Q|}\int_QM_{\Psi}(f\chi_Q)dx\simeq \|f\|_{\Psi^{\star},Q}.
\end{equation}

\begin{proof}[Proof of Theorem \ref{mainr}, part (ii)] This is just a combination of several previously established
bounds. As in the proof of the first part of Theorem \ref{mainr}, by Lemma \ref{kest}, it is enough to get a uniform estimate of
$\|{\mathcal A}_{\Phi, \mathcal S}f\|_{L^p(w)}$.

By the assumption $\Phi(t)\simeq t\int_1^t\frac{\Psi(u)}{u^2}du$ for $t\ge c_0$ we have that $\Phi\simeq\Psi^{\star}$,
where $\Psi^{\star}$ is defined by (\ref{psistar}). Hence, using (\ref{equiv}), we obtain
\begin{eqnarray*}
{\mathcal A}_{\Phi, \mathcal S}f(x)=\sum_{Q\in {\mathcal S}}\|f\|_{\Phi, \bar Q}\chi_Q(x)&\simeq& \sum_{Q\in {\mathcal S}}
\left(\frac{1}{|\bar Q|}\int_{\bar Q}M_{\Psi}(f\chi_{\bar Q})dx\right)\chi_Q(x)\\
&\le& {\mathcal T}(M_{\Psi}f)(x),
\end{eqnarray*}
where the operator ${\mathcal T}$ is defined by
$${\mathcal T}f(x)=\sum_{Q\in {\mathcal S}}\left(\frac{1}{|\bar Q|}\int_{\bar Q}fdx\right)\chi_Q(x).$$

Therefore, using that $\|{\mathcal T}\|_{L^p(w)}\lesssim [w]_{A_p}^{\max(1,1/(p-1))}$ (see \cite{CMP}) and applying Theorem \ref{buck},
we obtain
$$
\|{\mathcal A}_{\Phi, \mathcal S}\|_{L^p(w)}\lesssim \|{\mathcal T}\|_{L^p(w)}\|M_{\Psi}\|_{L^p(w)}
\lesssim [w]_{A_p}^{\max\big(1,\frac{1}{p-1}\big)}[w]_{A_p}^{\frac{1}{p}}C_{\Psi^{p-\e}}(p),
$$
where $\e\simeq [w]_{A_p}^{1-p'}$, and this completes the proof.
\end{proof}

\section{Remarks and complements}
\subsection{More about $A_p$ bounds for $\|{\mathcal C}\|_{L^p(w)}$}
Let $\a_p$ be the best possible exponent in
\begin{equation}\label{bestp}
\|{\mathcal C}\|_{L^p(w)}\lesssim [w]_{A_p}^{\a_p}.
\end{equation}

As we have seen, our proof of Corollary~\ref{carl}, part (ii), is
based essentially on (\ref{cond}) with
$$\Phi(t)=t\log({\rm e}+t)\log\log\log({\rm e}^{{\rm e}^{\rm e}}+t),$$
which is intimately related to Antonov's theorem \cite{A} on a.e. convergence of Fourier series for functions in $L\log L\log\log\log L$.
A question whether the class $L\log L\log\log\log L$ can be improved is still open. The main conjecture about this says that Fourier series
converge a.e. for functions in $L\log L$. A natural reformulation of this conjecture is that $(\ref{cond})$ for ${\mathcal C}$ holds with
$\Phi(t)=t\log ({\rm e}+t)$. Let us check what can be done assuming that this result is true.

First, it is easy to see that following our approach we would obtain that for all $p>1$,
\begin{equation}\label{pos1}
\|{\mathcal C}\|_{L^p(w)}\le c\frac{p^3}{(p-1)^2}[w]_{A_1}
\end{equation}
and
\begin{equation}\label{posb}
\|{\mathcal C}\|_{L^p(w)}\le c(p)[w]_{A_p}^{\max\big(p',\frac{2}{p-1}\big)}.
\end{equation}

In particular, the ``$L\log L$ conjecture" implies $\|{\mathcal C}\|_{L^p}\lesssim \frac{p^3}{(p-1)^2}$ and
$\a_p\le \max\big(p',\frac{2}{p-1}\big)$. It is natural to conjecture further that the unweighted bound for $\|{\mathcal C}\|_{L^p}$
is best possible, that is, $\|{\mathcal C}\|_{L^p}\simeq \frac{p^3}{(p-1)^2}$.
Then one can easily get a lower bound for $\a_p$ that coincides with the upper bound for $1<p\le 2$.

Indeed, a well known argument given by Fefferman-Pipher \cite{FP} (see also \cite{LPR} for an extension of this argument) says that if $T$ satisfies $\|T\|_{L^{p_0}(w)}\lesssim N([w]_{A_1})$ for some $p_0$, then $\|T\|_{L^r}\lesssim N(cr)$ as $r\to \infty$. Hence, on one hand, since $\|{\mathcal C}\|_{L^r}\simeq r$ as $r\to \infty$,
we obtain that $\a_p\ge 1$ for all $p>1$. On the other hand, let
${\mathcal C}_{\xi(\cdot)}$ be a linearization of ${\mathcal C}$ as in Remark~\ref{rem5}.
Then, by duality and by~(\ref{bestp}),
$$\|{\mathcal C}_{\xi(\cdot)}^*\|_{L^{p'}(w)}=\|{\mathcal C}_{\xi(\cdot)}\|_{L^{p}(w^{-(p-1)})}\lesssim [w^{-(p-1)}]_{A_{p}}^{\a_p}=[w]_{A_{p'}}^{\a_p(p-1)},$$
and hence $\|{\mathcal C}_{\xi(\cdot)}^*\|_{L^r}\lesssim r^{\a_p(p-1)}$ as $r\to \infty$, which implies
$$\|{\mathcal C}\|_{L^r}\lesssim \frac{1}{(r-1)^{\a_p(p-1)}}$$
as $r\to 1$. Conjecturing that $\|{\mathcal C}\|_{L^r}\simeq \frac{1}{(r-1)^2}$ as $r\to 1$, we obtain $\a_p\ge \frac{2}{p-1}$.
Therefore,  $\a_p\ge \max\big(1,\frac{2}{p-1}\big)$.

Concluding, we see that if the ``$L\log L$ conjecture" holds and if the best possible behavior of $\|{\mathcal C}\|_{L^p}$ is $\frac{1}{(p-1)^2}$ when $p$ is close to 1,
then for all $p>1$,
$$\max\Big(1,\frac{2}{p-1}\Big)\le \a_p\le \max\Big(p',\frac{2}{p-1}\Big).$$
In particular, $\a_p=\frac{2}{p-1}$ for $1<p\le 2$.

It seems that a natural obstacle in our approach is that the ``local mean oscillation estimate" essentially relies on the end-point information
of a given operator, while a sharp end-point information of the Carleson operator is currently unknown. It is natural to ask whether there is an
approach to sharp $L^p(w)$ estimates avoiding the information about end-point bounds. Observe that this is unknown even for Calder\'on-Zygmund operators.

\subsection{On mixed $A_p$-$A_{\infty}$ bounds}
Following recent works, where the $A_p$ bounds were improved by mixed $A_p$-$A_{\infty}$ bounds (see, e.g., \cite{HL,HP,HPR}),
we can give similar results for $T^{\mathcal F}$.

Given a weight $w$, define its $A_{\infty}$ constant by
$$[w]_{A_{\infty}}=\sup_{Q}\frac{1}{w(Q)}\int_{Q}M(w\chi_Q)dx.$$
It was shown in \cite{HP} that part (i) of Proposition \ref{sum} holds with the $[w]_{A_1}$ constant replaced by
$[w]_{A_{\infty}}$. Changing only this point in the proof of Theorem \ref{mainr}, part (i), we get
that for any $w\in A_1$ and for all $p>1$,
$$\|T^{\mathcal F}f\|_{L^p(w)}\le c(n,T)pC_{\Phi}\Big(\frac{p+1}{2}\Big)[w]_{A_1}^{\frac{1}{p}}[w]_{A_{\infty}}^{\frac{1}{p'}}\|f\|_{L^p(w)}.$$
For Calder\'on-Zygmund operators this inequality was obtained in \cite{HP}.

Further, it was shown in \cite{HPR} that if $w\in A_p$ and $\e\simeq [\si]_{A_{\infty}}$, where, as usual, $\si=w^{-\frac{1}{p-1}}$, then $w\in A_{p-\e}$ and $[w]_{A_{p-\e}}\lesssim [w]_{A_p}$.
It is easy to see from this result that the condition $\e\simeq [w]_{A_p}^{1-p'}$ in Theorem \ref{mainr} can be replaced by $\e\simeq [\si]_{A_{\infty}}$.

Then, in the case of the Carleson operator, by Remark \ref{rem3},
$$C_{\Psi^{p-\e}}(p)\simeq \frac{1}{\e^{1/p}}\log\log({\rm e}^{\rm e}+1/\e)\simeq [\si]_{A_{\infty}}^{\frac{1}{p}}\log\log({\rm e}^{\rm e}+[\si]_{A_{\infty}}),$$
and hence
$$\|M_{\Psi}\|_{L^p(w)}\lesssim\big([w]_{A_p}[\si]_{A_{\infty}}\big)^{\frac{1}{p}}\log\log({\rm e}^{\rm e}+[\si]_{A_{\infty}}).$$
Also, observe that the operator ${\mathcal T}$ defined in the proof of Theorem \ref{mainr} satisfies (see \cite{HL})
$$\|{\mathcal T}\|_{L^p(w)}\lesssim [w]_{A_p}^{\frac{1}{p}}\big([w]_{A_{\infty}}^{\frac{1}{p'}}+[\si]_{A_{\infty}}^{\frac{1}{p}}\big).$$
Therefore, combining this with the bound for $M_{\Psi}$, we obtain
$$\|{\mathcal C}\|_{L^p(w)}\lesssim [w]_{A_p}^{\frac{2}{p}}
\big([w]_{A_{\infty}}^{\frac{1}{p'}}+[\si]_{A_{\infty}}^{\frac{1}{p}}\big)
[\si]_{A_{\infty}}^{\frac{1}{p}}\log\log({\rm e}^{\rm e}+[\si]_{A_{\infty}}).
$$

\end{document}